\documentclass{amsart}     
\usepackage{pinlabel}  
\usepackage{graphicx}  

\usepackage{mathrsfs}

\title{Ends of Schreier graphs of hyperbolic groups}

\author{Audrey Vonseel}
\address{Universit\'{e} de Strasbourg, CNRS, IRMA UMR 7501, F-67000 Strasbourg, France}
\email{vonseel@math.unistra.fr}
\urladdr{}

\newtheorem{thm}{Theorem}[section]    
\newtheorem{lem}[thm]{Lemma}          
\newtheorem{prop}[thm]{Proposition}		
\newtheorem{cor}[thm]{Corollary}		
\theoremstyle{definition}
\newtheorem{defn}[thm]{Definition}    
\newtheorem{remark}[thm]{Remark}  		
\newcommand{\xbar}{\overline{x}}
\newcommand{\ybar}{\overline{y}}
\newcommand{\zbar}{\overline{z}}

\newcommand{\xzbar}{\overline{x_0}}
\newcommand{\cbxz}[1]{\bar{B}\left(\xzbar, #1\right)}
\newcommand{\obxz}[1]{B\left(\xzbar, #1\right)}
\newcommand{\xh}{X/H}
\newcommand{\gh}{G/H}
\newcommand{\caygs}{\Gamma_{\mathcal{S}}(G)}
\newcommand{\schghs}{\Gamma_{\mathcal{S}}(G,H)}
\newcommand{\coco}{C(\Lambda H)/H}
\newcommand{\dx}{\delta_{\scriptscriptstyle X}}
\newcommand{\dxh}{\delta_{\scriptscriptstyle X/H}}
\newcommand{\sdw}[1]{\mathscr{S}\left( #1 \right)}
\newcommand{\sph}[1]{S(\xzbar ,#1)}
\newcommand{\sphx}[1]{S(x_0,#1)}
\newcommand{\proj}[1]{\pi_{\scriptscriptstyle #1}}
\newcommand{\sdwxh}{\mathscr{S}_{\scriptscriptstyle X/H}}
\newcommand{\wch}{C(\Lambda H)}
\def\co{\colon\thinspace}

\begin{document}

\begin{abstract}    

We study the number of ends of a Schreier graph of a hyperbolic group. Let $G$ be a hyperbolic group and let $H$ be a subgroup of $G$. In general, there is no algorithm to compute the number of ends of a Schreier graph of the pair $(G,H)$. However, assuming that $H$ is a quasi-convex subgroup of $G$, we construct an algorithm. 
%
\end{abstract}

\maketitle


\section{Introduction}

The theory of ends has been introduced by H.Freundenthal for topological spaces \cite{Freudenthal31}. The ends are often depicted as the connected components of some boundary of the space where each end indicates a distinct way to move to infinity within the space. Later, H.Freudenthal extended his concept to finitely generated groups \cite{Freudenthal45}. This theory has been well studied in the middle of the last century.

For instance, H.Hopf \cite{Hopf44} showed that finitely generated groups have $0$, $1$, $2$ or infinitely many ends and later, J.Stallings \cite{Stallings71} proved that finitely generated groups have more than one end if and only if they split either as a free product with amalgamation or as a HNN extension over a finite group. 

Meanwhile, A.Borel adds a new perscepive to the theory of ends by studying ends of a group relatively to a subgroup \cite{Borel53}. If $G$ is a finitely generated group and $H$ a subgroup, the \emph{number of relative ends} of the pair $(G,H)$ is the number of ends of a corresponding Schreier graph, that is the quotient of a Cayley graph of $G$ under the action of $H$. Later, C.Houghton \cite{Houghton74} and then P.Scott \cite{Scott77} pursued this work.

The theory of ends remains of great interest. In 1999, V.Gerasimov studies connectedness of the boundary of hyperbolic groups with an algorithmic approach and obtains the following theorem: 
\begin{thm} \cite{Gerasimov99}
There is an algorithm that, given a finite presentation of a hyperbolic group, computes the number of ends of this group. \qed
\end{thm}

The question of an analogous result for relative ends of a pair of groups arise naturally. Surprisingly, the answer isn't straightforward. Indeed, we prove the following result: 
\begin{thm} \label{thmintro1}
There exist pairs of groups $(G,H)$, where $G$ is a hyperbolic group, for which it is impossible to decide algorithmically if the pair has 0, 2 or infinitely many relative ends. 
\end{thm}

Let us take a look at one of the few significant examples that can be clearly outlined. Consider a closed oriented surface $S$ of genus $g \geq 2$ endowed with a hyperbolic metric and denote by $G = \pi_1(S)$ the surface group.
Consider also a subsurface $\Sigma$ of $S$ with totally geodesic boundary.
The universal cover of the surface $S$ with base-point in $\Sigma$ is  $\mathbb{H}^2$. In this space, the connected component of the universal cover containing a lift of this base-point is a convex polygon $\mathcal{C}$ having infinitely many sides (see Figure \ref{fexrevetementuniversel}). If $H$ denotes the fundamental group of $\Sigma$, the quotient $\mathcal{C} / H$ is isomorphic to $\Sigma$.

\begin{figure}[ht!]
\labellist
\pinlabel $\mathcal{C}$ at 157 150
\endlabellist
\centering
\includegraphics[scale=0.3]{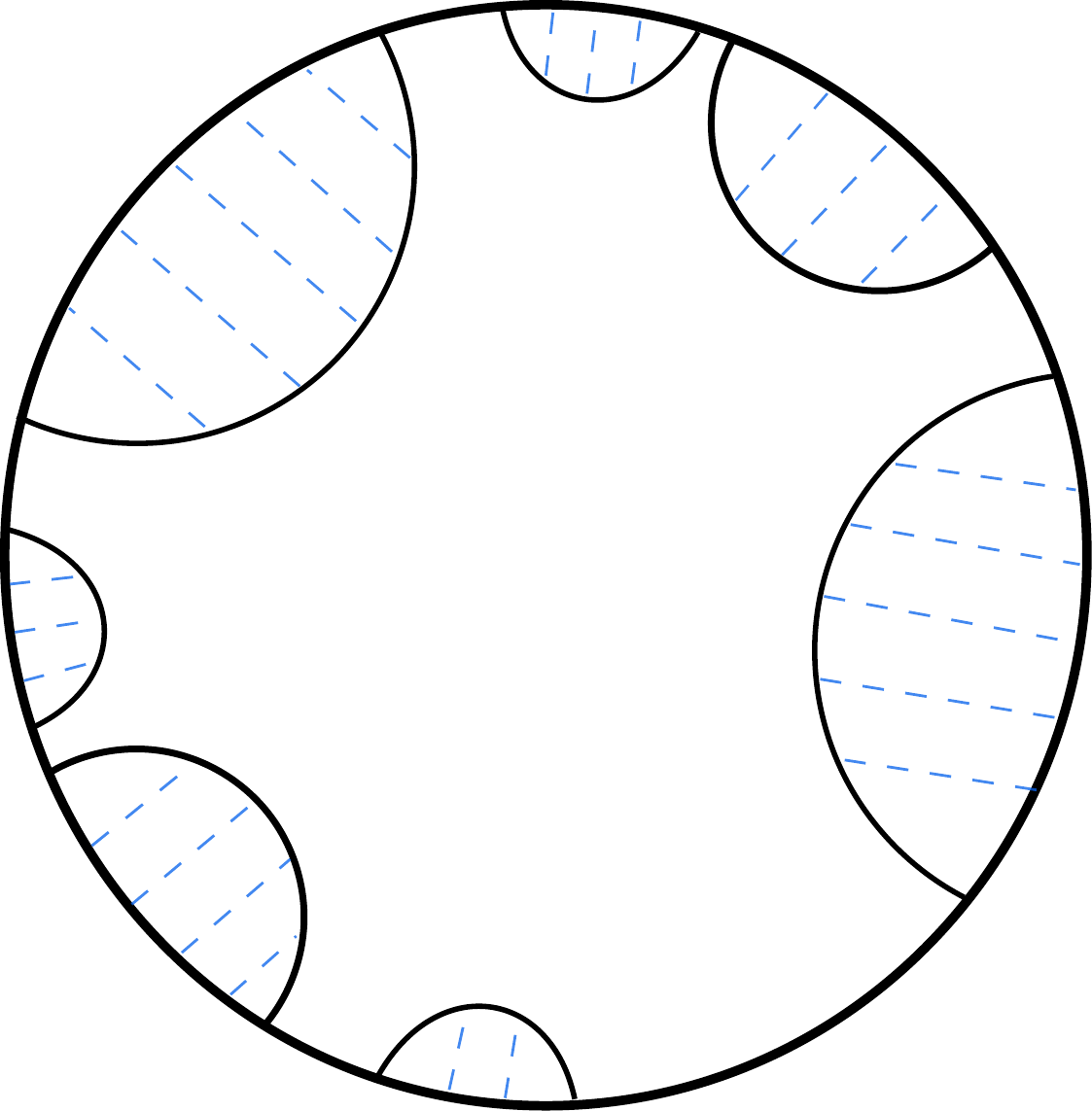}
\caption{The convex set $\mathcal{C}$ in $\mathbb{H}^2$ and a projection onto $\mathcal{C}$}
\label{fexrevetementuniversel}
\end{figure}

Note that there exists an orthogonal $H$-equivariant projection of $\mathbb{H}^2$ onto the convex polygon $\mathcal{C}$ \cite[II.2]{BridsonHaefliger99}. 
The equivariance induces that geodesic segments joining points of $\mathbb{H}^2$ to their projection onto the convex set $\mathcal{C}$ are all distinct modulo the action of $H$.
Thus the space formed by taking the quotient of the hyperbolic plane $\mathbb{H}^2$ by the group $H$ is composed of a convex core $\Sigma$ with a funnel attached to each connected component of its boundary (see Figure \ref{fexsurface2}). This space inherited from $\mathbb{H}^2$ a projection onto its convex core which, extend to infinity, identifies the connected components of the boundary of $\Sigma$ to the ends of the quotient space. 

\begin{figure}[ht!]
\labellist
\small \hair 2pt
\pinlabel $S$ at 5 100
\pinlabel $\Sigma$ at 155 100
\pinlabel $\Sigma$ at 465 100
\endlabellist
\centering
\includegraphics[scale=0.65]{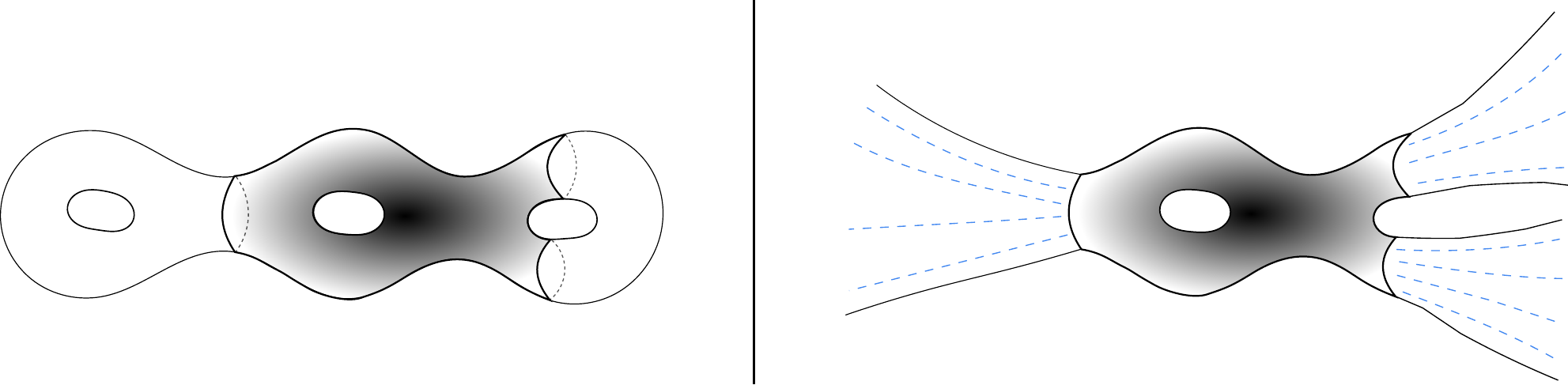}
\caption{Example of a closed surface $S$ of genus $3$ and a subsurface $\Sigma$ with totally geodesic boundary}
\label{fexsurface2}
\end{figure}

So the number of relative ends of the pair $(G,H)$ is equal to the number of connected components of the boundary of the convex core $\Sigma$. 
\medskip

In the light of this example, detecting the number of relative ends of a pair of group may still be possible in specific cases. Recall that the number of relative ends of a pair of group is essentially the number of ends of some quotient space. It appears that for the pairs of groups referred to in Theorem \ref{thmintro1}, the subgroup is usually not quasi-convex (Remark \ref{r_qc}). 
So we restrict ourselves to the study of the quotient of a geodesic proper hyperbolic space $X$ by a quasi-convex-cocompact group $H$ of isometries of $X$.  
The quotient space $\xh$ has good properties as it is hyperbolic and satisfies a geodesic extension property.
A careful study of the ends of this quotient space $\xh$ will lead us to the main result of this paper, namely:

\begin{thm} \label{t_main}
There exists an algorithm to compute the number of ends of a one-ended hyperbolic group relatively to a quasi-convex subgroup.
\end{thm}

At first, to study the ends of the former quotient space $\xh$, we identify a convex core and fix a base point $\xzbar$ in it. Given the previous example of the surface group, the geometry of the boundary of the convex core may help us determine the number of ends of $\xh$. So we determine a number $R_0$ such that the sphere $\sph{R_0}$ contains the convex core and we define an equivalence relation $\sim$ on this sphere (Definition \ref{d_eq_sphere}). 
In section \ref{s_ends_quotient_spaces}, we establish a bijection between the set of ends of $\xh$ and the set of equivalence classes for the relation $\sim$ on $\sph{R_0}$, via what comes to be known as their shadows (Definition \ref{d_sdw}).  
Therefore, the number of ends of $\xh$ is equal to the number of equivalence classes on $\sph{R_0}$.
Applying this result to group theory, the number of relative ends of a one-ended hyperbolic group and a quasi-convex subgroup is equal to the number of equivalence classes of the sphere $\sph{R_0}$ in the corresponding quotient space. For this reason, the construction of an algorithm to compute the number of this equivalence classes leads to Theorem \ref{t_main}.
\medskip

The section \ref{s_background} gathers definitions and properties dealing with hyperbolicity and quasi-convexity and recalls the graph-theoretic approach of the theory of ends. 
In the section \ref{s_quasiconvex hypothesis}, Markov properties and Rips construction are combined to prove that, in general, there is no algorithm to compute the number of relative ends of a pair of group.
The section \ref{s_quotient_spaces} presents the quotient spaces and some of their interesting properties, including a new proof of the hyperbolicity of this space.
The section \ref{s_ends_quotient_spaces} is devoted to the study of the ends of the quotient spaces where we relate connected components of the boundary of a convex core to the ends. 
Finally, the section \ref{s_application} presents an algorithm to compute the number of relative ends of a hyperbolic group and a quasi-convex subgroup.
\medskip

\textbf{Acknowledgement.}
This article exposes results from my thesis written at the University of Strasbourg. I am grateful to my Ph.D. advisor Thomas Delzant for his continual support and helpful advice during this work.
I would like to thank Indira Chatterji and Eric Swenson, my Ph.D. thesis reviewers, for their numerous questions and comments which greatly contributed to improve the presentation of this paper.

\section{Background}
\label{s_background}

This section introduces the notations and gathers basic facts on hyperbolic geometry. For more details, the reader can refer to \cite{Gromov87} but also \cite{CDP90}, \cite{GdH90} and \cite{BridsonHaefliger99}.

\subsection{Hyperbolic spaces and groups}
Recall some definitions of the original paper of M.Gromov \cite{Gromov87}. 
Let $(X,d)$ be a geodesic metric space and fix a base-point $x_0$ in $X$. The \emph{Gromov product} of two points $x,y \in X$ at $x_0$ is given by 
\begin{equation*}
\langle x,y \rangle_{x_0} = \frac{1}{2} \left( d(x_0,x) + d(x_0,y) - d(x,y) \right) .
\end{equation*}
The space $X$ is \emph{hyperbolic} if there exists a real number $\delta \geq 0$ such that for all points $x,y,z\in X$,
\begin{equation*}
\langle x,y\rangle_{x_0} \geq \min \{ \langle x,z \rangle_{x_0} , \langle y,z \rangle_{x_0} \} - \delta .
\end{equation*} 
It turns out that a different choice of base-point only changes the value of the constant $\delta$, as explained page $2$ of \cite{CDP90}.

There are various equivalent definition of hyperbolicity (see \cite[\S 6.3]{Gromov87}, \cite[1.3.6]{CDP90}). Thus, a geodesic metric space is also \emph{hyperbolic} if there exists $\delta \geq 0$ such that each side of a geodesic triangle is contained in the $\delta$-neighbourhood of the union of the two other sides. 
Such triangles are called \emph{$\delta$-thin}.

Moreover, the paragraph $8$ of \cite{Gromov87} provides a classification of isometries of hyperbolic spaces. In particular, \emph{hyperbolic isometries} of a hyperbolic space $X$ are elements $h$ such that for any point $x \in X$, the map $n \mapsto h^n x$ is a quasi-isometric embedding of $\mathbb{Z}$ into $X$.

\begin{remark} \label{rgeometrically}
We use here the convention of I.Kapovich and N.Benakli \cite[2.22]{KapovichBenakli02}: a group acts \emph{geometrically} on a geodesic space if it acts by isometries, cocompactly and properly discontinuously (for any compact $K$ in the space, there is a finite number of element $g$ of this group such that $K$ encounter $g\cdot K$). In particular, the \v{S}varc-Milnor lemma \cite{Svarc55}-\cite{Milnor68} asserts that if a group acts geometrically on a proper geodesic metric space then the orbit map is a quasi-isometry.
\end{remark}

Recall that a metric space is \emph{proper} if every closed ball is compact. 
A geodesic proper hyperbolic space $X$ can be compactified by attaching its boundary $\partial X$. To do so, say that two rays $c_1$ and $c_2$ in $X$ are \emph{equivalent} if there exists $K>0$ such that for all $t \geq 0$, $d(c_1(t),c_2(t)) \leq K$. Then the \emph{boundary} $\partial X$ of $X$ is the set of equivalence classes of geodesic rays in $X$ for this relation.
The space $X\cup \partial X$ is endowed with the induced topology from $X$ (see e.g. Chapter $7$ of \cite{GdH90}).
For a broad overview on boundaries of hyperbolic spaces we refer the reader to the paper \cite{KapovichBenakli02} of I.Kapovich and N.Benakli.

According to \cite[10.6.6]{CDP90}, every hyperbolic isometry fixes two points on the boundary $\partial X$, which are attractive and repulsive points for the action of this isometry on $\partial X$.
Moreover, in a $\delta$-hyperbolic space, geodesic triangles with vertices on the boundary are $4\delta$-thin as proven e.g. in \cite[3.9]{KapovichWeidmann04}; these triangles are often called \emph{ideal triangles}. 

The approximating tree technique provided by the following lemma is often hide behind powerful results (see Proposition \ref{pcontraction}, Theorem \ref{txhhyp}) as it transcribes properties from trees (where $\delta = 0$) to hyperbolic spaces.
\begin{lem}[{\cite[\S 6.1]{Gromov87}}] \label{lapproxtree}
Let $(X,d)$ be a $\delta$-hyperbolic space, $x_0,x_1, \ldots, x_n$ be a set of $n+1$ points of $X\cup \partial X$ and $Y$ be the union of $n$ geodesic segments joining $x_0$ to other points. Assume that $2n \leq 2^k +1$. There exist a simplicial tree $(T,d_T)$ and a continuous map $f \co X \to T$ such that:
\begin{enumerate}
\item For all $u,v \in Y$, $d(u,v) - 2k\delta  \leq d_T(f(u),f(v)) \leq d(u,v)$.
\item The restriction of $f$ to each segment is isometric. \qed
\end{enumerate}
\end{lem}

Hyperbolic spaces also satisfy the following geodesic extension property:
\begin{prop}\cite[3.1]{BestvinaMess91} \label{geodextprop}
Let $X$ be a geodesic proper $\delta$-hyperbolic space with base-point $x_0$ and $\# \partial X \geq 2$. Let $G$ be a group acting geometrically on $X$.
There exists a constant $\mu \geq 0$ such that for any point $x$ in $X$, there is a geodesic ray arising from $x_0$ passing within distance $\mu$ from $x$. \qed
\end{prop}

The constant $\mu$ is called a \emph{constant of geodesic extension} for $X$.

\begin{remark} \label{remarkgeodextprop}
If $x$ is a point in $X$ and $c \co \mathbb{R}_{\geq 0} \to X$ is a geodesic ray arising from $x_0$ and passing within distance $\mu$ from $x$ then, the Proposition 1.15 of \cite[III.H]{BridsonHaefliger99} gives $d(x,c(d(x_0,x))) \leq 2 \mu + 2\delta$ with $\delta$ a hyperbolicity constant for $X$. 
\end{remark}

A finitely generated group $G$ is \emph{hyperbolic} if, for a finite generating system $\mathcal{S}$ of $G$, the Cayley graph $\caygs$ is hyperbolic. This definition turns out to be independent on the choice of the finite generating set as a change of generating set produces a Cayley graph quasi-isometric to the first (see \cite[\S 2.3.E]{Gromov87}).
A hyperbolic group $G$ acts naturally by left multiplication on its Cayley graph $\caygs$ and this action continuously extends to an action of $G$ on the boundary of $\caygs$; then the \emph{boundary} of the hyperbolic group $G$ is the boundary of $\caygs$. As a quasi-isometry between hyperbolic spaces induces a homeomorphism between their boundary (see \cite[3.2.2]{CDP90}), this definition does not depend on the generating set up to homeomorphism.

\subsection{Quasi-convexity}
Let $(X,d)$ be a geodesic proper metric space. A subset $Y$ of $X$ is \emph{quasi-convex} if there exists a constant $\varepsilon \geq 0$ such that all geodesics joining two points of $Y$ in $X$ is contained in $Y^{+\varepsilon} = \{ x \in X \mid d(x,Y) \leq \varepsilon \}$, the $\varepsilon$-neighbourhood of $Y$. 
According to \cite[\S 7.3.A]{Gromov87}, if $X$ is a $\delta$-hyperbolic space and $Y \subset X$ is a $\varepsilon$-quasi-convex subset then $Y^{+\rho}$ is $2\delta$-quasi-convex, for all $\rho \geq \varepsilon$.
For instance, every quasi-geodesic is a quasi-convex set according to the Theorem of stability \cite[\S 7.2]{Gromov87}.

The concept of quasi-convexity extends naturally to groups: a subgroup $H$ of a finitely generated group $G$ is \emph{quasi-convex} if for some (and all) finite generating set $\mathcal{S}$ of $G$, $H$ is a quasi-convex subset of $\caygs$. 

Quasi-convex subgroups of hyperbolic groups have interesting properties: for instance, a quasi-convex subgroup of a hyperbolic group is in turn of finitely generated, hyperbolic and so finitely presentable (see chapter 10 of \cite{CDP90}).
\medskip

If the space $X$ is $\delta$-hyperbolic, recall that the limit set of a subgroup $H \leq Isom(X)$ is the set $\Lambda H = \overline{H \cdot x} \cap \partial X$ of accumulation points in $\partial X$ of an orbit of a point $x \in X$ under the action of $H$.
Then the \emph{weak convex hull} of the limit set of $H$, denoted by $C(\Lambda H)$, is the union of all bi-infinite geodesics with endpoints in $\Lambda H$. In particular, $\wch$ is $8\delta$-quasi-convex.

The group $H$ is \emph{quasi-convex-cocompact} if it acts geometrically on $\wch$. In particular, in this case, every orbit in $X$ under the action of $H$ is quasi-convex. 
\medskip

Let $X$ be a geodesic proper metric space and $Y$ a subset of $X$. 
Let $\eta \geq 0$. An \emph{$\eta$-projection} is a map $\pi \co X \to Y$ such that every point $x\in X$ satisfies $d(x,y) \leq d(x,Y) + \eta$. 
A fundamental property of projections on quasi-convex sets is the following contraction property:
\begin{prop} \cite{Coornaert89} \label{pcontraction}
Let $X$ be a $\delta$-hyperbolic geodesic space and $Y$ an $\varepsilon$-quasi-convex subset of $X$. Let $\pi \co X \to Y$ be an $\eta$-projection. For all points $x,x' \in X$, we have
\begin{equation*}
d(\pi(x),\pi(x')) \leq \max \{ \tau , 2\tau + d(x,x') - d(x,\pi(x)) - d(x',\pi(x')) \}
\end{equation*}
where $\tau = 12\delta + 2\varepsilon + 2\eta$.
Furthermore, we also have
\begin{equation*}
d(\pi(x),\pi(x')) \leq \tau + d(x,x') .
\end{equation*} \qed
\end{prop}

If $H$ is a group acting on $X$, it will be convenient to assume that $\varepsilon$-projections on $H$-invariant sets are also $H$-equivariant, as in the proof of Proposition \ref{palpha} for example.

\subsection{Ends and relative ends}

This part recalls the graph-theoretic approach used to define ends of finitely generated group in \cite[p.144-148]{BridsonHaefliger99}.

A map between two topological spaces is \emph{proper} if the inverse image of any compact set by this map is also a compact set. Two proper rays $c$ and $c'$ in a topological space $X$ are \emph{converging to the same end} of $X$ if for every compact set $K \subset X$ there exists an integer $N$ such that $c\left(\left[N,+\infty\right)\right)$ and $c'\left(\left[N,+\infty\right)\right)$ are contained in the same path-connected component of $X\setminus{K}$. It defines an equivalence relation on proper rays. The set of these equivalence classes form the \emph{set of ends} of the topological space $X$, denoted by $Ends(X)$. 
As a quasi-isometry between geodesic proper spaces induces a homeomorphism between their set of ends, the set of ends of a finitely generated group is the set of ends of a Cayley graph of this group. The \emph{number of ends} of a finitely generated group $G$ is denoted by $e(G)$.

\medskip

Let $G$ be a group generated by a finite set $\mathcal{S}$ and let $H$ be a subgroup of $G$. The \emph{Schreier graph} $\schghs$ of $G$ with respect to $H$, also called relative Cayley graph, is the quotient of $\caygs$ under the left action of $H$. 

\begin{remark} \label{rkschreier}
If $H$ is a normal subgroup of $G$, the Schreier graph $\schghs$ is exactly the Cayley graph of the group $G/H$ associated with $\mathcal{S}$. Indeed, in this case, the cosets of $H$ in $G$ are the orbits of elements of $G$ under the action of $H$.
\end{remark}

\medskip

As already mentioned in the introduction, the \emph{number of relative ends of the pair $(G,H)$ associated with $\mathcal{S}$} is the number of ends of the Schreier graph $\schghs$. The aim of the following work is to give an algorithm to determine the number of relative ends of a pair of groups. 

\section{The quasiconvex hypothesis}
\label{s_quasiconvex hypothesis}

Computing the number of relative ends of a pair of groups happens to be more difficult than expected. The following theorem states that, in general, there is no algorithm to compute the number of relative ends of a pair of groups.

\begin{thm} \label{tsansquasi-convexepasdAlgorithme}
We can construct a pair of groups $(G,H)$ for which it is impossible to decide algorithmically if:
\begin{enumerate}
\item $\gh$ is finite;
\item the pair $(G,H)$ has two relative ends;
\item the pair $(G,H)$ has infinitely many relative ends.
\end{enumerate}
\end{thm}

The proof of this theorem mainly relies on the Rips construction given in \cite{Rips82}. Given a finite presentation of a group $Q$, the Rips construction furnishes a hyperbolic group $G$ (in fact a small cancellation group) and a subgroup $H$ of $G$ for which the number of relative ends depends on non recursively recognisable properties of the original group presentation, namely Markov properties.
\medskip

As explained in \cite[IV.4]{LS77}, a \emph{Markov property} of a finitely presented group is a property $\mathcal{P}$ for which there exist two finitely presented groups, $G_+$ and $G_-$, such that $G_+$ satisfies $\mathcal{P}$ and $G_-$ cannot be embedded in any finitely presented group which satisfies $\mathcal{P}$.
A large range of well known properties of finitely presented groups are Markov properties (see also \cite{Miller92} for more details). 
The following Markov properties are used in the proof of Theorem \ref{tsansquasi-convexepasdAlgorithme}:
\begin{enumerate}
\item ``Being the trivial group" (i.e. the group reduced to a single element) with $G_+ = 1$ and $G_- = \mathbb{Z}_2$ as witnesses.
\item ``Being a group of cardinality $2$" with $G_+ = \mathbb{Z}_2$ and $G_- = \mathbb{Z}_5$.
\item ``Being of cardinality less than or equal to $2$" with $G_+ = \mathbb{Z}_2$ and $G_- = \mathbb{Z}_5$. 
\end{enumerate}

The main result about Markov properties is the Adyan-Rabin theorem \cite{Adyan55} \cite{Rabin58}. It states that Markov properties are not recursively recognizable. This will be a key point of the proof of Theorem \ref{tsansquasi-convexepasdAlgorithme}.

We will also need the following lemma:

\begin{lem} \label{lQ*Z2}
Let $A$ be a finitely presented group. Let $Q = A \ast \mathbb{Z}_2$. 
\begin{enumerate}
\item The group $Q$ is finite if and only if the group $A$ is the trivial group;
\item The group $Q$ has $2$ ends if and only if the group $A$ is of cardinality $2$;
\item The group $Q$ has infinitely many ends if and only if the cardinality of $A$ is greater than $2$. 
\end{enumerate}
\end{lem} 

\begin{proof} 
Firstly, by Hopf's theorem, $e(Q)=0$ if and only if the group $Q$ is finite. 
But by definition, the free product of two groups is finite if and only if one of them is the trivial group. This implies that the group $Q= A \ast \mathbb{Z}_2$ is finite if and only if the group $A$ is the trivial group. Then $e(Q)=0$ if and only if $A$ is the trivial group.
\medskip

The second point of this result comes from the Theorem 5.12 of \cite{ScottWall79} of P.Scott and T.Wall which indicates that $e(Q)=2$ if and only if $Q$ splits into a HNN extension $A*_C$ or in a free product with amalgamation $A *_C B$ where $C$ is finite, $|A/C| = |B/C| = 2$. By comparison with $Q = A \ast \mathbb{Z}_2$, the group $C$ is here the trivial group, $B = \mathbb{Z}_2$ and $A$ is a group of cardinal $2$.
Thus the group $Q$ has two ends if and only if $Q = \mathbb{Z}_2 \ast \mathbb{Z}_2$.
\medskip

The point (1) and Stallings' theorem \cite{Stallings71} imply that $Q$ has infinitely many ends if and only if the cardinality of $A$ is greater than $2$.
\end{proof}

\begin{proof}[Proof of Theorem \ref{tsansquasi-convexepasdAlgorithme}]
Let $A$ be a finitely presented group. 
By Adyan-Rabin's theorem, there is no algorithm to decide if $A$ satisfies the properties :
\begin{enumerate}
\item ``Being the trivial group";
\item ``Being a group of cardinality $2$";
\item ``Being of cardinality less than or equal to $2$".
\end{enumerate}
Then, by Lemma \ref{lQ*Z2}, there is no algorithm to decide if 
\begin{enumerate}
\item $Q$ is finite;
\item $Q$ has $2$ ends;
\item $Q$ has infinitely many ends.
\end{enumerate}

Now denote by $G$ and $H$ groups associated to $Q$ by the Rips construction. So the number of ends of $Q$ is equal to the number of relative ends of the pair $(G,H)$.
Indeed, the Rips construction for the group $Q$ furnishes a short exact sequence 
\begin{equation*}
1 \to H \to G \to Q \to 1 
\end{equation*}
where $G$ is a hyperbolic group and $H$ is a finitely presented subgroup of $G$. In particular, the group $G/H$ is isomorphic to $Q$. Therefore the Schreier graph of the pair $(G,H)$ is exactly the Cayley graph of $Q$ according to the same generating set. Thus the number of relative ends of the pair $(G,H)$ is equal to the number of ends of $Q$. In particular, there is no algorithm to decide if:
\begin{enumerate}
\item $\gh$ is finite;
\item The pair $(G,H)$ has two relative ends;
\item The pair $(G,H)$ has infinitely many relative ends.
\end{enumerate}
\end{proof}

It is possible to give an explicit finite group presentation $\langle \mathcal{S} \mid \mathcal{R} \rangle$ for which we cannot decide if the group defined is the trivial group (see \cite{Collins86} for instance). 
This example provides a group $Q$ from which the Rips construction produces explicitly a hyperbolic group $G$ and a subgroup $H$ such that $\gh$ is finite.
And so this is an example for the case (1) of the theorem. 
Likewise, with the presentations $\langle \mathcal{S}, b \mid \mathcal{R} , b^2 \rangle$ and $\langle \mathcal{S}, b \mid \mathcal{R} , b^3 \rangle$, we can explicit hyperbolic groups and a subgroup such that the number of relative ends is two or infinitely many. 

\begin{remark} \label{r_qc}
In each case, the subgroup $H$ raised by the Rips construction is a normal subgroup of the hyperbolic group $G$. If such a subgroup was quasi-convex then it would necessarily be finite or of finite index in $G$ (see Proposition 3.9 in \cite{ABC91}). Added to the well known properties of quasi-convex subgroups of hyperbolic group, this observation tends to consider this particular class of subgroups.
\end{remark}

\section{Quotient spaces}
\label{s_quotient_spaces}
\subsection{Description of quotient spaces}

From now on, let $X$ be a geodesic proper hyperbolic space and let $H$ be a quasi-convex-cocompact group of isometries of $X$. Denote by $\wch$ the weak convex hull of the limit of $X$ and fix a point $x_0 \in \wch$. Let $\dx \geq 0$ be a constant large enough for $X$ to be $\dx$-hyperbolic and any orbit of $H$ in $X$ to be $\dx$-quasi-convex.

For every point $x$ in $X$, denote by $\xbar$ the point in $\xh$ equal to the orbit of $x$ under the action of $H$.
The quotient space $\xh$ is endowed by a natural distance inherited from $X$: for every points $x$ and $y$ of $X$, 
\begin{align*}
d_{\xh} (\xbar,\ybar) = \inf\limits_{h\in H} d_X (x,hy) = d_X(x,H\cdot y) .
\end{align*}
(Distance's subscripts are omitted when it is clear from the context.)

As $H$ is quasi-convex-cocompact, the action of $H$ on its weak convex hull $\wch$ is geometric and so the quotient $\coco$, called \emph{convex core}, is contained in a closed ball of $\xh$. 

\subsection{Hyperbolicity}

In order to get the following result, we will restrict ourselves to the case where $H$ contains no element of finite order.

\begin{prop} \label{palpha}
There is a constant $\alpha$, depending only on $\dx$, such that for every element $h \in H$ and every point $x \in X$ within distance $\alpha$ from $H \cdot x_0$, we have $d(x,hx) \geq 100\dx$.
\end{prop}

\begin{proof}
Let $h$ be an element of $H$ and $x$ be a point in $X$.
Denote by $\pi$ a $H$-equivariant $\dx$-projection of $X$ on the $\dx$-quasi-convex subset $H \cdot x_0$ of $X$. 

When the translation length $[h] := \inf_{x \in X} d(x,hx)$ of $h$ is greater than $16\dx$, the contraction property of projections on quasi-convex sets (Proposition \ref{pcontraction}) asserts that 
\begin{equation*}
d(\pi(x),h\pi(x)) \leq 32 \dx + d(x,hx) - d(x ,\pi(x)) - d(hx,\pi(hx)) .
\end{equation*}
If $d(x,hx)< 100\dx$, this implies that $d(x,H\cdot x_0) < 132 \dx$.
Consequently, if the point $x$ is at distance at least $132 \dx$ from the quasi-convex $H\cdot x_0$, the distance between $x$ and $hx$ is greater than $100 \dx$.
\medskip

Assume now that the translation length of $h$ is less than $16 \dx$.  
According to \cite[\S 8.5]{Gromov87}, there exists an integer $n_0$ such that, if $h^-$ and $h^+$ are respectively repulsive and attractive fixed points for $h$, then $h^{n_0}$ fixes a geodesic line joining $h^-$ to $h^+$.
By taking a multiple of $n_0$ if necessary, we can assume that $[h^{n_0}]$ is greater than $16\dx$. 
Thus the previous inequalities are valid for $h^{n_0}$. 

To conclude, the proposition holds with $\alpha := (132 + 100n_0) \dx$.
\end{proof}

The following theorem was stated in \cite[5.3]{Gromov87} without proof. Later, it has been proved by I.Kapovich in the article \cite{Kapovich02} and by R.Foord \cite{Foord00}, both unpublished.

\begin{thm} \label{txhhyp}
Let $X$ be a geodesic proper hyperbolic space and $H$ be quasi-convex-cocompact group of isometries of $X$. The space $\xh$ is a hyperbolic space.
\end{thm}

\begin{proof}
We intend to find a constant $\dxh$ such that every point $\xbar$, $\ybar$ and $\zbar$ in $\xh$ satisfy
\begin{equation}
\label{eqnhyp}
\tag{$\star$}
\langle \ybar,\zbar \rangle_{\overline{x_0}} \geq \min \{ \langle \ybar,\xbar \rangle_{\overline{x_0}} , \langle \xbar,\zbar \rangle_{\overline{x_0}} \} - \dxh .
\end{equation}
To do so, set $\rho := diam (\coco) + \alpha+ \dx$ where $diam(\coco)$ is the diameter of the convex core $\coco$ and $\alpha$ is the constant arising in Proposition \ref{palpha}. Fix three points $\xbar$, $\ybar$ and $\zbar$ in $\xh$.

Firstly, assume that $\min \{ \langle \ybar,\xbar \rangle_{\overline{x_0}} , \langle \xbar,\zbar \rangle_{\overline{x_0}} \} > \rho$. Let $x$ be a lift of $\xbar$ in $X$ satisfying $d(x_0,x)=d(\xzbar,\xbar)$. Denote by $y$ and $z$ respective lifts of $\ybar$ and $\zbar$ such that $d(x,y)=d(\xbar,\ybar)$ and $d(x,z)=d(\xbar,\zbar)$. 
By definition of the Gromov product, these points satisfy
\begin{equation*}
\langle y,x\rangle_{x_0}  \geq \langle \ybar,\xbar \rangle_{\overline{x_0}}  > \rho
\qquad \textrm{ and } \qquad
\langle x,z\rangle_{x_0}  \geq \langle \xbar,\zbar \rangle_{\overline{x_0}} > \rho .
\end{equation*}
Then, as $X$ is $\dx$-hyperbolic, the following inequality holds:
\begin{equation*}
\langle y,z \rangle_{x_0} \geq \min \{ \langle \ybar,\xbar \rangle_{\overline{x_0}} , \langle \xbar,\zbar \rangle_{\overline{x_0}} \} - \dx .
\end{equation*}
Since $\rho \geq \dx$, the set $(H \cdot x_0)^{+\rho}$ is $2\dx$-quasi-convex and there exists a $\dx$-projection $\pi \co X \to (H \cdot x_0)^{+\rho}$ on this set. 

Let $x_1$ denotes the point on a geodesic segment $[x_0,x]$ at distance $\rho$ from $x_0$. Apply the approximating tree lemma \ref{lapproxtree} to a geodesic triangle $\Delta$ with vertices $x$, $x_1$ and $\pi(x)$: there exists a simplicial tree $T$ endowed with a simplicial metric $d_T$ such that (for convenience, we denote by the same letter points in $\Delta$ and their image in $T$) for all $u,v \in \Delta$, $d(u,v) -4\dx \leq d_T(u,v) \leq d(u,v)$ and  $d_T(x_1,\pi(x)) = d(x_1,\pi(x))$.

If $x'$ denotes a projection of $x$ on the segment $[x_1,\pi(x)]$ in $T$, the approximating tree lemma gives:
\begin{align*}
d_T(x,x') & \geq d(x,x') - 4 \dx \\
& \geq d(x,(H\cdot x_0)^{+\rho}) - 6\dx \, \textrm{ since } (H\cdot x_0)^{+\rho} \textrm{ is } 2\dx \textrm{-quasi-convex,} \\
& \geq d(x,\pi(x)) - 7\dx \, \textrm{ since } \pi \textrm{ is a } \dx \textrm{-projection on } (H\cdot x_0)^{+\rho} , \\
& \geq d_T(x,\pi(x)) - 7\dx.
\end{align*}
And so $d_T(\pi(x),x') = d_T(x,\pi(x)) - d_T(x,x') \leq 7\dx$.
A similar computation with $x_1$ instead of $\pi(x)$ also gives $d_T(x_1,x') = d_T(x,x_1) - d_T(x,x') \leq 7\dx$.
Then $d(x_1,\pi(x)) = d_T(x_1,\pi(x)) \leq 14\dx$ and so
\begin{equation*}
d(x_0,\pi(x)) \leq \rho + 14\dx.
\end{equation*}
The contraction property of projections on quasi-convex sets indicates that
\begin{itemize}
\item either $d(\pi(x),\pi(y)) \leq 18\dx $ and the above inequality induces  $d(x_0,\pi(y)) \leq \rho + 32\dx$.
\item either $d(\pi(x),\pi(y)) \leq 36\dx + d(x,y) - d(x,\pi(x)) - d(y,\pi(y))$ which is equivalent to 
\begin{align*}
d(x,y) & \geq d(\pi(y),x_0) - d(x_0, \pi(x)) + d(x,x_0) - d(x_0,\pi(x)) + d(y,\pi(y)) - 36\dx \\
& \geq d(\pi(y),x_0) + d(x,x_0) + d(y,\pi(y)) - 2\rho - 64\dx .
\end{align*}
The choice of the lifts $y$, $z$ and the definition of a $\dx$-projection imply that
\begin{equation*}
d(\xbar,\ybar) \geq d(\pi(y),x_0) + d(\xbar,\xzbar) + d(\ybar,\xzbar) - 3\rho - 64\dx .
\end{equation*}
Then the triangular inequality for $d(\xbar,\ybar)$ induces $d(x_0,\pi(y)) \leq 3\rho + 64\dx $. 
\end{itemize}
As $\pi$ is a $\dx$-projection, the point $y$ satisfies
\begin{equation*}
d(\ybar,\xzbar) - \rho \leq d(y,\pi(y)) \leq d(\ybar,\xzbar) - \rho + \dx .
\end{equation*}
In every instance, we have 
\begin{align*}
d(x_0,y) & \leq d(x_0, \pi(y)) + d(\pi(y),y) \\
& \leq  3\rho + 64\dx + d(\xzbar,\ybar) -\rho + \dx .
\end{align*}
Replacing $y$ by $z$ in the previous lines gives $d(x_0,z) \leq 2\rho + 65\dx + d(\xzbar,\zbar)$. Then the Gromov product of $y$ and $z$ at $x_0$ is equal to 
\begin{align*}
\langle y,z \rangle_{x_0} & = \dfrac{1}{2} \left( d(x_0,y) + d(x_0,z) - d(y,z) \right)\\
& \leq \dfrac{1}{2} \left( d(\xzbar,\ybar) + d(\xzbar,\zbar) - d(\ybar,\zbar) \right) + 2\rho + 65\dx \\
& \leq  \langle \ybar,\zbar \rangle_{\overline{x_0}} +2\rho + 65\dx .
\end{align*}
This implies that if $\min \{ \langle \ybar,\xbar \rangle_{\overline{x_0}} , \langle \xbar,\zbar \rangle_{\overline{x_0}} \} > \rho$, the points $\xbar$, $\ybar$ and $\zbar$ satisfy the inequality \eqref{eqnhyp}.

Furthermore, if $\min \{ \langle \ybar,\xbar \rangle_{\overline{x_0}} , \langle \xbar,\zbar \rangle_{\overline{x_0}} \} \leq \rho$ then $\min \{ \langle \ybar,\xbar \rangle_{\overline{x_0}} , \langle \xbar,\zbar \rangle_{\overline{x_0}} \}  - (2\rho + 65\dx) < 0$. As the Gromov product $\langle \ybar,\zbar \rangle_{\overline{x_0}}$ is always positive, the inequality \eqref{eqnhyp} is still satisfied.
Therefore, we can conclude that the space $\xh$ is $\dxh$-hyperbolic with $\dxh := 2~ (diam (\coco) + \alpha + \varepsilon) + 65\dx$. 
\end{proof}

\subsection{Covering and geodesic extension property}

\begin{prop}\label{pcover}
The space $X \setminus{(H\cdot x_0)^{+\dxh}}$ is a covering space of the complement of the closed ball $\cbxz{\dxh}$ in $\xh$.
\end{prop}

\begin{proof}
In light of Proposition \ref{palpha} and Theorem \ref{txhhyp}, for every element $h\in H$ and every point $x \in X$ with ${d(x,H\cdot x_0) > \dxh}$, there exists $d >0$ such that the closed ball $\cbxz{d}$ never intersects its translated $h \cdot \cbxz{d}$.
By a classical result of topology (see \cite[1.40]{Hatcher}), the quotient map from $X \setminus{(H\cdot x_0)^{+\dxh}}$ to $(X \setminus{(H\cdot x_0)^{+\dxh}})/H$ is then a covering map. 
Then, by definition of the distance in $\xh$, the quotient $(H\cdot x_0)^{+\dxh}/H$ is equal to the ball $\cbxz{\dxh}$. Therefore, $X \setminus{(H\cdot x_0)^{+\dxh}}$ is locally homeomorphic to the complement of $\cbxz{\dxh}$ in $\xh$.
\end{proof}

This proposition allows to prove that the quotient space $\xh$ also satisfies the geodesic extension property of Proposition \ref{geodextprop}. 

\begin{cor}\label{cgeodext}
If $\dx$ is also a constant of geodesic extension for $X$, there exists $\mu \geq 0$ such that for every point $\xbar$ in $\xh$, there is a geodesic ray arising from $\xzbar$ passing within distance $\mu$ from $x$. 
\end{cor}

\begin{proof}
By Theorem \ref{txhhyp}, there exists some constant $\dxh$ such that $\xh$ is $\dxh$-hyperbolic. Consider a point $\xbar$ in $\xh$ in the complement of $B(\xzbar,\dxh)$ and denote by $x$ a lift of $\xbar$ in $X$ such that $d(x_0,x)=d(\xzbar,\xbar)$.
By assumption, there exists a geodesic ray $c \co \mathbb{R}_{\geq 0} \to X$ arising from $x_0$ passing within distance $\dx$ from $x$.

Using the covering map described in Proposition \ref{pcover}, the ray $c$ furnishes a geodesic ray $\overline{c} \co [\dxh+1, +\infty) \to \xh$ arising from the point $\overline{c(\dxh+1)}$ and passing within distance $\dx$ from $\xbar$. This ray together with a geodesic ray arising from $\xzbar$ with same endpoint at infinity and a geodesic segment $[\xzbar,\overline{c(\dxh+1)}]$ form an ideal triangle in $\xh$.
As ideal triangles are $4\dxh$-thin in $\xh$, the point $\xbar$ within distance $4\dxh+\dx$ from a geodesic ray arising from $\xzbar$.
So $4 \dxh + \dx$ is a suitable constant of geodesic extension for $\xh$.
\end{proof}

By replacing the constant $\dxh$ by $4\dxh+\dx$ if necessary, we can always assume that $\xh$ satisfies the geodesic extension property with $\dxh$.

\section{Ends of quotient spaces}
\label{s_ends_quotient_spaces}

\subsection{Boundary}

By Theorem \ref{txhhyp}, the quotient space $\xh$ is hyperbolic. This means that we can use the among of results on hyperbolic spaces to obtain informations on its ends. In particular, the following result relates the set of ends to the boundary of hyperbolic spaces. 

\begin{prop}\cite[5.17]{GdH90}
Let $X$ be a geodesic proper hyperbolic space. 
The natural map from the boundary $\partial X$ of $X$ to the set of ends $Ends(X)$ of $X$ is continuous and surjective and the fibres are connected components of $\partial X$.  \qed
\end{prop}

Therefore, studying the boundary of the quotient space can give us informations on its space of ends. 

\begin{prop} \label{pcovmap}
If $H$ acts properly discontinuously on $X$ then the quotient map $\partial X \setminus{\Lambda H} \to \partial \xh$ is a covering map.
\end{prop}

\begin{proof}
According to the theorem of M.Coornaert proven in \cite{Coornaert89}, if the group $H$ acts properly discontinuously on $X$, it also acts properly discontinuously on $\partial X \setminus{\Lambda H}$. Moreover, as $H$ is torsion free, its action is also free. This implies that for every point $u \in \partial X \setminus{\Lambda H}$, there exists an open set $U$ containing $u$ such that for every non trivial element $h \in H$, $U$ does not intersect $h\cdot U$. 
By a classical result of topology (see \cite[I.3.Ex.23]{Hatcher}), it follows that the quotient map $\partial X \setminus{\Lambda H} \to (\partial X \setminus{\Lambda H})/H$ is a covering map.
The equality $(\partial X \setminus{\Lambda H})/H = \partial \xh$ arises naturally.
\end{proof}

\begin{prop} \label{p_bordxh}
If $\partial X$ is connected then $\partial \xh$ is locally path connected. 
\end{prop}

The proof of this proposition relies on results discussed in detail in section \ref{s_application}.

\begin{proof}
A theorem proven independently by B.Bowditch \cite{Bowditch99} and G.A.Swarup \cite{Swarup96} asserts that if $\partial X$ is connected then it has no global cut point. Moreover, according to the \cite[\S 3]{BestvinaMess91} indicates that if $\partial X$ has no global cut point, it is locally path connected. The combination of these results with Proposition \ref{pcovmap} completes the proof.
\end{proof}

Although this result isn't sufficient to understand the set of ends of $\xh$, it brings into view the interesting work of M.Bestvina and G.Mess.

\subsection{Large Bestvina-Mess condition}
\label{ss_BM}

In the paper \cite{BestvinaMess91}, M.Bestvina and G.Mess give a condition $(\ddag_M)$ on $X$ that implies the local connectedness of boundary of $X$. For our purpose, we are going to consider a slight variation of their condition: for $M>K>0$, set

\begin{tabular}{l p{300pt}}
$(\ddag_{M,K})$ & There exists an integer $L > 0$ such that for every $R \geq K + 2\dx$ and for all points $x \in \sphx{R}$ and $y \in \sphx{R}^{+K}$ such that $d(x,y) \leq M$, there is a path of length less than or equal to $L$ joining $x$ to $y$ in the complement of  $\bar{B}(x_0,R - K - 2\dx)$. \tabularnewline
\end{tabular}

This condition also characterizes the local connectedness of one-ended hyperbolic spaces. More specifically, the proof of M.Bestvina and G.Mess in \cite[3.2]{BestvinaMess91} also gives literally the following result:
\begin{prop} \label{pddagboundarylocallyconencted}
If there exist constants $M,K$ with $M \geq 4K + 18\dx + 6$ such that $X$ satisfies  the condition $(\ddag_{M,K})$ then the boundary of $X$ is locally connected. \qed
\end{prop}

Now we give a Bestvina-Mess condition for $\xh$: for $M>0$, set

\begin{tabular}{l p{300pt}}
$(\dag_{M})\,\,\,\,\,$ & There exists an integer $L > 0$ such that for every $R \geq \max\{ M + \dxh, 8 \dxh\}$ and for all points $\xbar,\ybar \in \sph{R}$ and $d(\xbar,\ybar) \leq M$, there is a path of length less than or equal to $L$ joining $\xbar$ to $\ybar$ in the complement of  $\cbxz{R - 8\dxh}$. \tabularnewline
\end{tabular}

The following result outlines a link between the condition $(\ddag_{M,K})$ for $X$ and the condition $(\dag_M)$ for $\xh$:

\begin{prop} \label{pXddagXHdag}
If there exists a constant $M>4\dxh$ such that $X$ satisfies the condition $(\ddag_{M,4\dxh})$ then the quotient space $\xh$ satisfies the condition $(\dag_M)$.
\end{prop}

\begin{proof}
Let us fix $M > 4\dxh$ such that $(\ddag_{M,4\dxh})$ is satisfied by $X$. Consider points $\xbar$, $\ybar$ in $\xh$ such that $d(\xzbar,\xbar) = d(\xzbar,\ybar) = R \geq \max\{ M + \dxh , 8 \dxh \}$ and $d(\xbar,\ybar) \leq M$. 
Denote by $x$ a lift of $\xbar$ in $X$ satisfying $d(x_0,x)=d(\xzbar,\xbar)$. As $R > M + \dxh$, lift $\bar{B}(\xbar,M)$ to $\bar{B}(x,M)$ using the covering map described in Proposition \ref{pcover}. By lifting a geodesic in this ball joining $\xbar$ to $\ybar$, we obtain a lift $y$ of $\ybar$ in the ball $\bar{B}(x,M)$ such that $d(x,y) = d(\xbar,\ybar)$. We are going to prove that these points $x$ and $y$ satisfy the conditions of $(\ddag_{M,4\dxh})$.

Let $\pi$ be a $\dx$-projection from $X$ to the $2\dx$-quasi-convex thickened orbit $(H\cdot x_0)^{+\dxh}$. In particular, $y$ satisfies
\begin{equation*}
d(\ybar,\xzbar) - \dxh \leq d(y,\pi(y)) \leq d(\ybar,\xzbar) - \dxh + \dx .
\end{equation*}
The contraction property of projections on quasi-convex sets (Proposition \ref{pcontraction}) asserts that
\begin{itemize}
\item either $d(\pi(x),\pi(y)) \leq 18 \dx$ and we have the following upper bound on $d(x_0,y)$:
\begin{align*}
d(x_0,y) & \leq d(x_0,\pi(x)) + d(\pi(x),\pi(y)) + d(\pi(y),y) \\
& \leq \dxh + 18\dx + d(\xzbar,\ybar) + \dx \\
& \leq 2\dxh + d(x_0,x) .
\end{align*}
\item or $d(\pi(x),\pi(y)) \leq 36 \dx + d(x,y) - d(x,\pi(x)) - d(y,\pi(y))$, which implies that
\begin{align*}
d(x,y) & \geq d(\pi(x),\pi(y)) + d(x,\pi(x)) + d(y,\pi(y)) - 36 \dx \\
& \geq d(x_0,\pi(y)) + d(x,x_0) + d(y,\pi(y)) - 2d(x_0,\pi(x)) - 36 \dx.
\end{align*}
Then our choice of $x$ and $y$ yields that 
\begin{equation*}
d(\xbar,\ybar) \geq d(x_0,\pi(y)) + d(\xzbar,\xbar) + d(\xzbar,\ybar) -3\dxh - 36 \dx .
\end{equation*}
The triangle inequality for $d(\xbar,\ybar)$ gives $d(x_0,\pi(y)) \leq 3\dxh + 36\dx$.
Therefore, we obtain
\begin{align*}
d(x_0,y) & \leq d(x_0, \pi(y)) + d(\pi(y),y) \\
& \leq 4\dxh + d(x_0,x) .
\end{align*}
\end{itemize}

In either case, $x$ and $y$ satisfy $d(x_0,x) = R$, $| d(x_0,x) - d(x_0,y) | \leq 4\dxh$ and $d(x,y) \leq M$. As $X$ satisfies $(\ddag_{M,4\dxh})$, there exists an integer $L > 0$ and a path of length less than or equal to $L$ joining $x$ to $y$ in the complement of $\bar{B}(x_0,R-4\dxh-2\dx)$.

Now, we show that the image of this path in $\xh$ stays in the complement of $\cbxz{R-8\dxh}$. To do so, consider a point $z\in X$ on this path. 
The above arguments also shows that the image $\zbar$ of $z$ in $\xh$ satisfies $d(x_0,z) \leq 3\dxh + 37 \dx + d(\xzbar,\zbar)$ and this implies that $d(\xzbar,\zbar) \geq R -8\dxh$.
Then the image of the path in $\xh$ forms a path of length less than or equal to $L$ joining $\xbar$ to $\ybar$ in the complement of $\cbxz{R-8\dxh}$.
\end{proof}

\subsection{From sphere to shadows}
From now on, fix an integer $M \geq 43 \dxh + 4$ and set $R_0 = M + \dxh$. In the text below, paths are supposed to be of finite length and $\sph{R_0}$ will denote the sphere at $\xzbar$ of radius $R_0$.

In light of the Bestvina-Mess conditions, the following equivalence relation arises naturally.

\begin{defn} \label{d_eq_sphere}
Two points on $\sph{R_0}$ are \emph{equivalent} if there exists a path joining them in the complement of the open ball $\obxz{R_0 - 3 \dxh}$. 
\end{defn}

Denote by $\sim$ this equivalence relation on the sphere $\sph{R_0}$ and $[ \xbar ]$ the class of a point $\xbar$ under this relation. 
In order to link these equivalence classes with the whole quotient space $\xh$, define a projection on $\cbxz{R_0}$. 
\medskip

First, note that the thickened orbit $(H \cdot x_0)^{+R_0}$ is $2\dx$-quasi-convex since $R_0$ is greater than the quasi-convexity constant of the orbit which is $\dx$. So consider a $H$-equivariant $\dx$-projection on this thickened orbit, namely $\proj{0} \co X \to (H \cdot x_0)^{+ R_0}$. Given $\proj{0}$, define a projection in $\xh$, that is a map $\proj{0}' \co \xh \to \cbxz{R_0}$ such that:
\begin{itemize}
\item For every point $\xbar$ in the complement of $\cbxz{R_0}$, denote by $x$ a lift of $\xbar$ in $X$. The map $\proj{0}'$ sends $\xbar$ to $\overline{\proj{0}(x)}$, the image of $\proj{0}(x)$ in $\xh$.
\item The map $\proj{0}'$ is the identity on the closed ball $\cbxz{R_0}$. 
\end{itemize}

\begin{remark} \label{rprojprim}
The definition of $\proj{0}'$ does not depend on the choice of the lift of $\xzbar$. Indeed, if $x'$ is an other lift of $\xbar$ in $X$, there exists an element $h$ in $H$ such that $x'=hx$. Then, by $H$-equivariance of $\proj{0}$, we have that
\begin{equation*}
d(\overline{\proj{0}(x)} , \overline{\proj{0}(x')} )  = \min \limits_{h \in H} d(\proj{0}(x) , h\proj{0}(h'x))  = \min \limits_{h \in H}  d(\proj{0}(x) , hh'\proj{0}(x) )  = 0.
\end{equation*}
The projection $\proj{0}'(\xbar)$ is also equal to $\overline{\proj{0}(x')}$.
\end{remark}

\begin{prop} \label{prayongeodmemeclasse}
For every geodesic ray  $c \co \mathbb{R}_{\scriptscriptstyle \geq 0} \to \xh$ arising from $\xzbar$ and for every integer $R \geq R_0$, there exists a path joining $c(R_0)$ to $\proj{0}'(c(R))$ in the complement of $\obxz{R_0 -10\dx}$. 
\end{prop}

\begin{proof}
Let $(r_i)_{1\leq i \leq n}$ be a finite sequence of integers such that $R_0 = r_1 \leq r_2 \leq \ldots \leq r_n = R$ and $r_{i+1}-r_i \leq \dx$ for all $i \in \{1,\ldots ,n-1\}$. Denote by $x_1$ a lift of $c(r_1)$ in $X$. Then, for each $i>1$, denote by $x_i$ a lift of $c(r_i)$ in $X$ satisfying $d(x_{i-1},x_{i}) \leq \dx$. 
By Remark \ref{rprojprim}, $d(\proj{0}'(c(r_i)),\proj{0}'(c(r_{i+1}))) = d( \overline{\proj{0}(x_i)} , \overline{\proj{0} (x_{i+1})} )$ for all $i \in \{1,\ldots ,n-1\}$.
Now apply the contraction property of projections on quasi-convex sets (Proposition \ref{pcontraction}) to each pair $(x_i,x_{i+1})$: for all $i \in \{1,\ldots ,n-1\}$, 
\begin{align*}
 d( \overline{\proj{0}(x_i)} , \overline{\proj{0} (x_{i+1})} ) & \leq d(\proj{0}(x_i) , \proj{0}(x_{i+1})) \\
& \leq 18 \dx + d(x_i , x_{i+1}) \\
& < 20 \dx .
\end{align*}
As $\proj{0}'$ is the identity on $\sph{R_0}$, we also have that $\proj{0}'(c(R_0)) = \overline{x_1} = c(R_0)$. 
Now consider the path obtained by concatenating a geodesic segment $[c(R_0) , \overline{\proj{0}(x_1)} ]$ with geodesic segments $[ \overline{\proj{0}(x_i)} , \overline{\proj{0}(x_{i+1})} ]$ for $i \in \{1,\ldots,n-1\}$ and $[ \overline{\proj{0}(x_n)}, \proj{0}'(c(R)) ]$. This path joins $c(R_0)$ to $\proj{0}'(c(R))$ and stays in the complement of the open ball $\obxz{R_0 -10\dx}$.
\end{proof}

As Theorem \ref{txhhyp} gives $\dxh$ much greater than $65\dx$, Proposition \ref{prayongeodmemeclasse} indicates that $c(R_0) \sim \proj{0}'(c(R))$ and it induces the following definition.

\begin{defn}\label{d_sdw}
A \emph{shadow} of an equivalence class for $\sim$ is the inverse image of this class under $\proj{0}'$.
\end{defn}

``Being in the same shadow" is clearly an equivalence relation on points in the complement of $\obxz{R_0}$ for which shadows are equivalence classes. Denote by $\sdw{\xbar}$ the shadow of the class $\left[\xbar\right]$ and $\sdwxh$ the set of shadows in $\xh$.

\begin{prop} \label{pnbrombres=nbrclasses}
The set $\sdwxh$ of shadows in $\xh$ is in bijection with the set $\sph{R_0}/\sim$ of equivalence classes for the relation $\sim$ on $\sph{R_0}$.
\end{prop}

\begin{proof}
Consider $f \co \sph{R_0} /\sim \, \to \sdwxh $ the map which sends each equivalence class to its shadow in $\xh$.

Let $\xbar$ and $\ybar$ be two points in $\sph{R_0}$ such that $\sdw{\xbar} = \sdw{\ybar}$. In particular, $\proj{0}'(\sdw{\xbar} ) = \proj{0}'(\sdw{\ybar})$. Moreover, by definition of $\proj{0}'$, the class of $\xbar$ for the relation $\sim$ is $[\xbar] = \proj{0}'(\sdw{\xbar} )$ and the class of $\ybar$ for this relation is $[\ybar] = \proj{0}'(\sdw{\ybar})$. Thus the map $f$ is one-to-one. 

Let $\mathscr{S} \in \sdwxh$. By definition of the shadow, there is a point $\xbar \in \sph{R_0}$ such that $\mathscr{S}$ is the inverse image of $[\xbar]$ under $\proj{0}'$. Then the image of $[\xbar]$ in $\sdwxh$ is $\mathscr{S}$ and the map $f$ is surjective.

To conclude, $f$ is a bijection between the set of equivalence classes on $\sph{R_0}$ and the set of shadows in $\xh$. 
\end{proof}

\subsection{From shadows to ends}

In this paragraph, we establish a link between shadows and ends of the quotient space $\xh$. 
Recall that $M \geq 43 \dxh + 4$ and $R_0 = M + \dxh$ are fixed.

\begin{lem} \label{lptsdansmemesphereetombrejointsparchemin}
If $X$ satisfies $(\ddag_{M,4\dxh})$, then for all integers $R \geq R_0$, any two points on $\sph{R}$ in the same shadow are joined by a path in the complement of the open ball $\obxz{R - 8\dxh}$.
\end{lem}

We follow closely the arguments given by M.Bestvina and G.Mess in \cite[3.2]{BestvinaMess91}.

\begin{proof}
We proceed by induction on $R$. Consider two points $\xbar$, $\ybar$ on $\sph{R_0}$ in the same shadow. This means that $\proj{0}'(\xbar)$ and $\proj{0}'(\ybar)$ are joined by a path in the complement of $\obxz{R_0 -3\dxh}$. But $\proj{0}'(\xbar) = \xbar$ and $\proj{0}'(\ybar)=\ybar$. So there exists a path joining $\xbar$ to $\ybar$ in the complement of  $\obxz{R_0 -3\dxh} \supset \obxz{R_0 - 8\dxh}$.

Let $R \geq R_0$. Assume now that all pair of points on $\sph{R}$ in the same shadow are joined by a path in the complement of $\obxz{R-8\dxh}$. Consider two points $\xbar$ and $\ybar$ on $\sph{R+1}$ in the same shadow. Denote by $c_0$ and $c_1$ two geodesic rays arising from $\xzbar$ passing within distance $\dxh$ from $\xbar$ and $\ybar$ respectively. 

To prove the inductive step, construct a sequence of points on $\sph{R+1}$ that are sufficiently close to use that $\xh$ satisfies $(\dag_M)$ and to join $\xbar$ and $\ybar$ by a path in the complement of $\obxz{R+1 - 8\dxh}$. 
To do so, we prove that $c_0(R)$ and $c_1(R)$ are in the same shadow using
\begin{equation*}
\proj{0}'(c_0(R)) \sim \proj{0}'(c_0(R+1)) \sim \proj{0}'(\xbar) \sim \proj{0}'(\ybar) \sim \proj{0}'(c_1(R+1)) \sim \proj{0}'(c_1(R)) .
\end{equation*}
and then construct the sequence of points using the induction hypothesis for $c_0(R)$ and $c_1(R)$.
\medskip

First of all, prove $\proj{0}'(c_0(R+1)) \sim \proj{0}'(\xbar)$. Denote by $z$ a lift of $c_0(R+1)$ and $x'$ a lift of $\xbar$ such that $d(x',z) = d(c_0(R+1),\xbar) \leq 4 \dxh$ (according to Remark \ref{remarkgeodextprop}). By contraction property (Proposition \ref{pcontraction}), we obtain
\begin{align*}
d(\proj{0}'(c_0(R+1)),\proj{0}'(\xbar)) & = d(\overline{\proj{0}(z)}, \overline{\proj{0}(x')}) \\
& \leq  d(\proj{0}(z), \proj{0}(x'))  \\
& \leq 18\dx + d(z,x') \\
& < 5\dxh .
\end{align*}

Therefore, there exists a path joining points $\proj{0}'(c_0(R+1))$ and $\proj{0}'(\xbar)$ in the complement of $\obxz{R_0 - 3\dxh}$. 
An analogous computation indicates that there exists also a path joining $\proj{0}'(c_1(R+1))$ and $\proj{0}'(\ybar)$ in the complement of  $\obxz{R_0 - 3\dxh}$. 
By Proposition \ref{prayongeodmemeclasse}, there is a path joining $\proj{0}'(c_i(R))$ to $\proj{0}'(c_i(R+1))$ in the complement of $\obxz{R_0 -10\dx}$ for $i = 0 , 1$. 
Moreover, as $\xbar$ and $\ybar$ are in the same shadow, there exists also a path joining $\proj{0}'(\xbar)$ and $\proj{0}'(\ybar)$ in the complement of $\obxz{R_0 - 3\dxh}$. 
By concatenating these paths, we obtain a path from $\proj{0}'(c_0(R))$ to $\proj{0}'(c_1(R))$ in the complement of  $\obxz{R_0 - 3\dxh}$. 
And so the points $c_0(R)$ and $c_1(R)$ on the sphere $\sph{R}$ are in the same shadow. 

Figure \ref{figsdwpath} depicts the inductive step.

\begin{figure}[ht!]
\labellist
\small \hair 2pt
\pinlabel $c_0$ at 55 300
\pinlabel $c_i$ at 210 325
\pinlabel $c_{i+1}$ at 300 325
\pinlabel $c_1$ at 490 300
\pinlabel $\xbar$ at 20 198
\pinlabel $\ybar$ at 460 175
\pinlabel $\overline{s_i}$ at 208 268
\pinlabel $\overline{s_{i+1}}$ at 293 255
\pinlabel $\overline{p_{i}}$ at 227 180
\pinlabel $\overline{p_{i+1}}$ at 280 92
\pinlabel $\overline{q_{i}}$ at 185 180
\pinlabel $\overline{q_{i+1}}$ at 242 95
\pinlabel $c_0(R)$ at 85 135
\pinlabel $c_0(R+1)$ at 85 225
\pinlabel $c_1(R)$ at 444 112
\pinlabel $c_1(R+1)$ at 479 200
\pinlabel $c_i(R+1)$ at 233 224
\pinlabel $c_{i+1}(R+1)$ at 310 220
\pinlabel $\sph{R+1}$ at 530 180
\pinlabel $\sph{R}$ at 500 117
\pinlabel $\sph{R-8\dxh}$ at 505 32
\endlabellist
\centering
\includegraphics[scale=0.62]{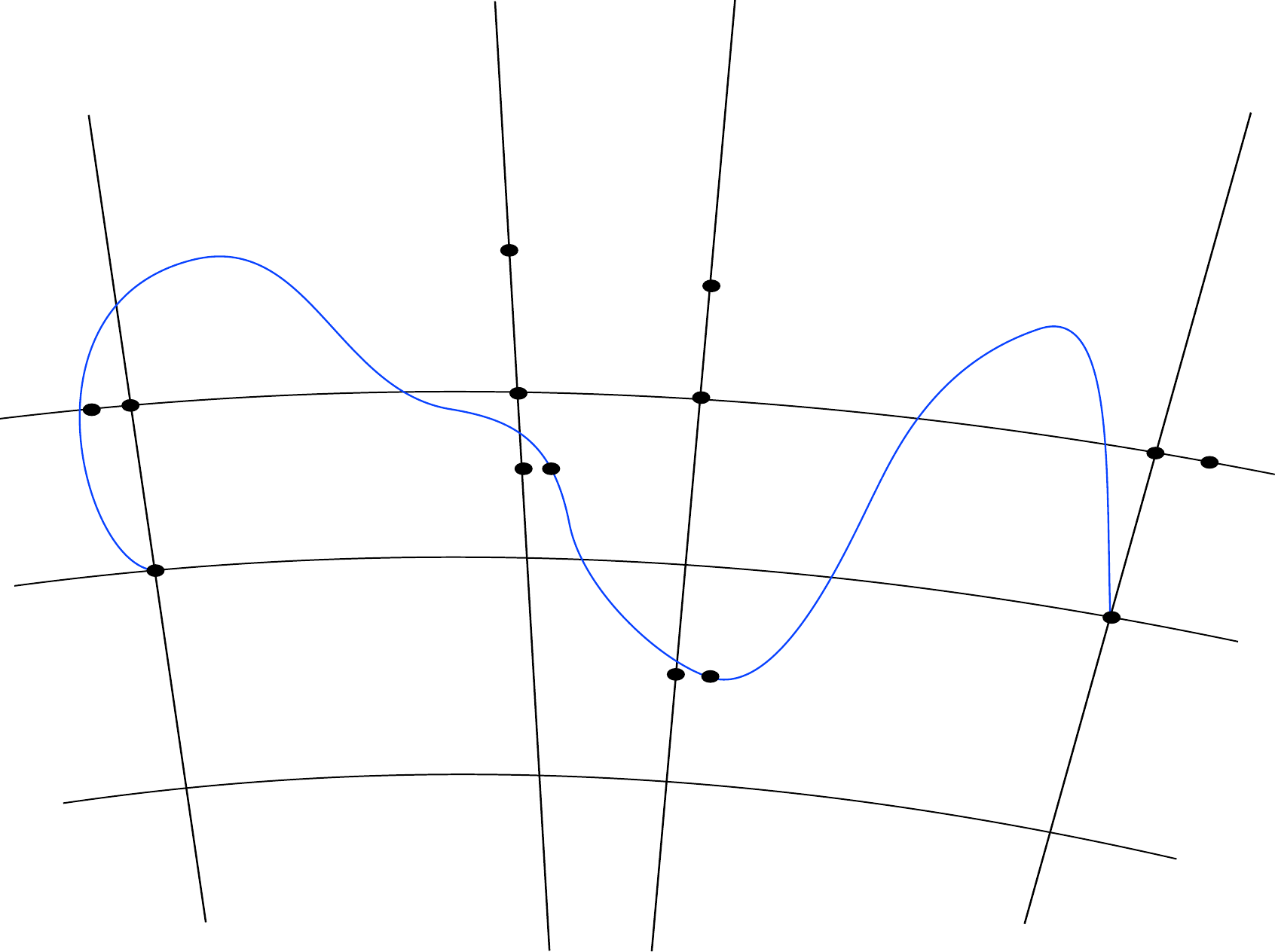}
\caption{Joining $\xbar$ and $\ybar$ in the complement of  $\obxz{R+1-8\dxh}$}
\label{figsdwpath}
\end{figure}

Now apply the induction hypothesis to $c_0(R)$ and $c_1(R)$: there exists a path joining $c_0(R)$ and $c_1(R)$ in the complement of $\obxz{R - 8\dxh}$. 
Consider $\overline{p_0}, \ldots ,\overline{p_n}$ points on the path satisfying $\overline{p_0} = c_0(R)$, $\overline{p_n} = c_1(R)$ and $d(\overline{p_i},\overline{p_{i+1}}) \leq \dx$ for all $i \in \{ 0, \ldots, n-1 \}$. Then for all $i$, denote by $c_{i}$ a geodesic ray arising from $\xzbar$ passing within distance $\dxh$ from $\overline{p_i}$. 

Denote by $\overline{q_i}$ a point of $c_{i}$ at distance less than or equal to $\dxh$ of $\overline{p_i}$. Then the distance between $\overline{q_i}$ and $\xzbar$ is at least $R- 9\dxh$. Denote by $\overline{s_i}$ a point of $c_i$ satisfying 
\begin{align*}
\begin{cases}
d(\overline{s_i},\xzbar) \geq R+1 ,\\
d(\overline{s_i},\overline{q_i}) \leq 1 + 9\dxh .
\end{cases}
\end{align*}
For all $i \in \{0, \ldots , n-1 \}$, these points satisfy $d(\overline{s_i},\overline{s_{i+1}}) \leq 2 + 20\dxh + \dx $. 
Consider a geodesic triangle formed by the sub-segments $[\xzbar , \overline{s_{i}}] \subset c_i$, $[\xzbar , \overline{s_{i+1}}] \subset c_{i+1}$ and a geodesic segment joining $\overline{s_{i}}$ and $\overline{s_{i+1}}$. According to \cite[III.H.1.15]{BridsonHaefliger99}, we obtain the following:
\begin{align*}
d(c_i(R+1),c_{i+1}(R+1)) & \leq 2 (2 + 20\dxh + \dx + \dxh) \\
& \leq 4 + 43 \dxh .
\end{align*}

As $X$ satisfies the property $(\ddag_{4+43\dxh ,4\dxh})$, there is a path joining $c_{i}(R+1)$ and $c_{i+1}(R+1)$ in the complement of $\obxz{R+1- 8\dxh}$ for all $i \in \{ 0, \ldots ,n-1 \}$. Moreover, given Remark \ref{remarkgeodextprop}, $d(c_0(R+1), \xbar) \leq 4\dxh$ and $d(c_1(R+1), \ybar) \leq 4\dxh$. Thus, there exists a path joining $\xbar$ to $\ybar$ in the complement of  $\obxz{R+1 - 8\dxh}$.
\medskip

To conclude, for all $R \geq R_0$, all two points on $\sph{R}$ in the same shadow are joined by a path in the complement of  $\obxz{R - 8\dxh}$.
\end{proof}

This lemma extends to a more general result:

\begin{prop} \label{pptsdansmemeombrejointsparchemin}
If $X$ satisfies $(\ddag_{M,4\dxh})$ then any two points $\xbar$, $\ybar \in \xh$ in the same shadow are joined by a path in the complement of $\obxz{\min\{ d(\xzbar,\xbar) , d(\xzbar,\ybar)\} - 8\dxh}$.
\end{prop}

\begin{proof}
Consider two points $\xbar$, $\ybar$ in $\xh$ in the same shadow. Set $R_x := d(\xzbar,\xbar)$ and $R_y := d(\xzbar,\ybar)$. 
By exchanging roles between $\xbar$ and $\ybar$ if necessary, we can assume that $R_x$ is less than or equal to $R_y$ and so $\min\{ d(\xzbar,\xbar) , d(\xzbar,\ybar)\} = R_x$. Denote by $c$ a geodesic ray arising from $\xzbar$ passing within distance $\dxh$ from $\ybar$ and set $\ybar' := c(R_x)$. 
By Proposition \ref{prayongeodmemeclasse} and as $d(\ybar,c) \leq \dxh$, $\ybar$ and $\ybar'$ are in the same shadow. This implies that $\xbar$ and $\ybar'$ are in the same shadow.
By applying Lemma \ref{lptsdansmemesphereetombrejointsparchemin}, there exists a path joining $\xbar$ to $\ybar'$ in the complement of $\obxz{R_x - 8\dxh}$. 
Moreover, in light of Remark \ref{remarkgeodextprop}, $d(\ybar,c(R_y)) \leq 4\dxh$ and so these points are joined by a path in the complement of  $\obxz{R_y - 2\dxh}$. 

The concatenation of these paths with the restriction of $c$ to $[R_x,R_y]$ forms a path joining $\xbar$ to $\ybar$ in the complement of $\obxz{R_x - 8\dxh}$.
\end{proof}

Recall that Proposition \ref{prayongeodmemeclasse} states that all points on a geodesic ray arising from $\xzbar$ at distance greater than or equal to $R_0$ from $\xzbar$ are in the same shadow. Therefore, our relation $\sim$ extends to geodesic rays in $\xh$ arising from $\xzbar$ in the following manner: two geodesic rays $c, c' \co \mathbb{R}_{\geq 0} \to \xh$ arising from $\xzbar$ are ``\emph{in the same shadow}" if all points on $c$ and $c'$ at distance greater than or equal to $R_0$ from $\xzbar$ are in the same shadow.
Denote by $\sdw{c} := \sdw{c(R_0)}$ the shadow defined by the class of $c(R_0)$. 

\begin{remark} \label{rsujectionrayonombre}
There is a natural surjection from the set of geodesic rays arising from $\xzbar$ to the set of shadows in $\xh$.
\end{remark}

\begin{lem} \label{lmemeombrememebout}
Assume that $X$ satisfies the property $(\ddag_{M,4\dxh})$. 
Two geodesic rays in $\xh$ are in the same shadow if and only if they converge to the same end of $\xh$.
\end{lem}

\begin{proof}
Consider two geodesic rays $c$ and $c'$ converging to the same end in $\xh$. This means that for all $R \geq R_0$, there is an integer $R' \geq R$ such that for all $r,r' \geq R'$, the points $c(r)$ and $c'(r')$ are in the same path connected component of the complement of $\obxz{R}$. Therefore, for all $r,r' \geq R'$, the points $c(r)$ and $c'(r')$ are joined by a path in the complement of  $\obxz{R} \supset \obxz{R_0}$. 
By Proposition \ref{prayongeodmemeclasse}, the following equivalences are satisfied for all $r,r' \geq R'$:
\begin{align*}
\proj{0}'(c(r)) & \sim  \proj{0}'(c(R_0)) = c(R_0) ,\\
\proj{0}'(c'(r')) & \sim  \proj{0}'(c'(R_0)) = c'(R_0) .
\end{align*}
This means that there exist paths joining $\proj{0}'(c(r))$ to $c(R_0)$ and $\proj{0}'(c'(r'))$ to $c'(R_0)$ in the complement of  $\obxz{R_0 - 8\dxh}$.
The concatenation of these paths with the restrictions of rays $c([R_0,r])$ and $c'([R_0,r'])$  gives a path joining $\proj{0}'(c(r))$ to $\proj{0}'(c'(r'))$ in the complement of  $\obxz{R_0 - 8\dxh}$. By Proposition \ref{prayongeodmemeclasse}, this fact extends to every point far away from $\xzbar$. Indeed, for all $r_1, r_2 \geq R_0$, we have the following:
\begin{equation*}
\proj{0}'(c(r_1)) \sim c(R_0) \sim  \proj{0}'(c(r)) \sim  \proj{0}'(c'(r')) \sim  c'(R_0) \sim  \proj{0}'(c'(r_2)) .
\end{equation*}
Therefore, for all $r_1, r_2 \geq R_0$, there is a path joining $\proj{0}'(c(r_1))$ to $\proj{0}'(c'(r_2))$ in the complement of  the open ball $\obxz{R_0 - 8\dxh}$ i.e. the geodesic rays $c$ and $c'$ are in the same shadow.

Now consider two geodesic rays $c, c' \co \mathbb{R}_{\geq 0} \to \xh$ in the same shadow. By Proposition \ref{pptsdansmemeombrejointsparchemin}, for all $r,r' \geq R_0$, there is a path joining $c(r)$ to $c'(r')$ in the complement of  $\obxz{\min\{r,r'\} - 8\dxh}$. This implies that for $R \geq R_0$, every point of $c([R,+\infty))$ is joined by a path in the complement of  $\obxz{R - 8\dxh}$ to a point of $c'([R,+\infty))$. Therefore these rays converge to the same end of $\xh$. 
\end{proof}

\begin{remark}
The fact that two geodesic rays in $\xh$ converging to a same end are in the same shadow holds even if $X$ does not satisfy $(\ddag_{M,4\dxh})$.
\end{remark}

To sum up, if the space $X$ satisfies the property $(\ddag_{M,4\dxh})$ with $M \geq 43\dxh + 4$, the number of shadows for geodesic rays in $\xh$ is equal to the number of ends for $\xh$. This implies the following result:

\begin{thm} \label{tbijectionbtxhclassesdeq}
Let $M \geq 43\dxh + 4$ and $R_0 \geq M + \dxh$. If $X$ satisfies the property $(\ddag_{M,4\dxh})$, then the set of ends of $\xh$ is in bijection with the set of equivalence classes for the relation $\sim$ on the sphere $\sph{R_0}$.
\end{thm}

\begin{proof}
Let $f \co Ends(\xh) \to \sdwxh$ be a map which sends an end $end(c)$ defined by a geodesic ray $c$ to the shadow $\sdw{c}$.

Let $c$ and $c'$ be two geodesic rays arising from $\xzbar$ in $\xh$ such that $\sdw{c} = \sdw{c'}$. Lemma \ref{lmemeombrememebout} implies then that $c$ and $c'$ converge to the same end; this means that $end(c) = end(c')$. Thus $f$ is one-to-one.

Let $\mathscr{S} \in \sdwxh$. By Remark \ref{rsujectionrayonombre}, the map sending geodesic rays arising from $\xzbar$ to shadows in $\xh$ is surjective. So there exists a geodesic ray $c$ arising from $\xzbar$ such that $\mathscr{S} = \sdw{c}$. The image of $end(c)$ under $f$ is then exactly $\mathscr{S}$. Therefore, the map $f$ is surjective.

So there is a bijection between the set of ends of $\xh$ and the set of shadows in $\xh$. By Proposition \ref{pnbrombres=nbrclasses}, there is a bijection between the set of ends and the set of equivalence classes for $\sim$ on the sphere $\sph{R_0}$ .
\end{proof}

In the light of Theorem \ref{tbijectionbtxhclassesdeq}, the following corollary is straightforward.

\begin{cor} \label{cnombreboutsdexh=nombreclasses}
Let $M \geq 43\dxh + 4$ and $R_0 \geq M + \dxh$. If the space $X$ satisfies the property $(\ddag_{M,4\dxh})$, then the number of ends of $\xh$ is equal to the number of equivalence classes for the relation $\sim$ on the sphere $\sph{R_0}$. \qed
\end{cor}

\begin{remark}
In particular, the number of equivalence classes on the sphere centered in $\xzbar$ of radius $R_0$ is finite since it is bounded by the size of this sphere. In this case, the number of ends of $\xh$ is necessarily finite.
\end{remark}

\section{Application to group theory}
\label{s_application}

In what follows, $G$ is a hyperbolic group  with connected boundary given by a finite presentation $\langle \mathcal{S} \mid \mathcal{R} \rangle$ and $X$ is the Cayley graph of $G$ with respect to $\mathcal{S}$. Moreover, assume that there exists a quasi-convex subgroup $H$ of $G$ and denote by $x_0$ a point in the weak convex hull $C(\Lambda H)$ of the limit set of $H$. In particular, the group $H$ is a quasi-convex-cocompact group of isometries of $X$. Then the number of relative ends of the pair $(G,H)$ is the number of ends of the associated Schreier graph $\xh$. Denote by $\dx$ a hyperbolicity constant for $X$ which is a geodesic extension constant for $X$.

The aim of this section is to give an algorithm to compute the number of relative ends of the pair $(G,H)$. 
\medskip

Here we go back over the results used in the proof of Proposition \ref{p_bordxh}. M.Bestvina and G.Mess proved that if a hyperbolic space does not satisfy their condition then its boundary has a cut-point (see Proposition 3.3 of \cite{BestvinaMess91}). Their proof remains valid under condition $(\ddag_{M,K})$ and gives the following result:

\begin{prop}\label{ppointdecoupureglobal}
If there exist constants $M> K >0$ such that $X$ does not satisfy the condition $(\ddag_{M,K})$ then the boundary of $X$ contains a global cut point. \qed
\end{prop}

A few years later, B.Bowditch and G.Swarup proved independently that there is no such cut point in one-ended hyperbolic groups (see \cite[0.3]{Bowditch99} and \cite{Swarup96}).

\begin{thm}\label{tbowditchswarup}
If $G$ is a one-ended hyperbolic group, then the boundary $\partial G$ of $G$ contains no global cut point. \qed
\end{thm}

In the light of Proposition \ref{ppointdecoupureglobal} and Theorem \ref{tbowditchswarup}, the following statement holds.

\begin{cor} \label{coneendedhypgrpsatisfddag}
If $G$ is a hyperbolic group with connected boundary, then any Cayley graph of $G$ satisfies the condition $(\ddag_{M,K})$ for all $M>K>0$. \qed
\end{cor}

The following result arises easily from Corollary \ref{coneendedhypgrpsatisfddag} and Corollary \ref{cnombreboutsdexh=nombreclasses} for the determination of relative ends. 

\begin{cor} \label{coneendedhypgrpnbrofends}
Under our assumption on $G$ and $H$, there exists a constant $R_0$ such that the number of relative ends of the pair $(G,H)$ is equal to the number of equivalence classes on $\sph{R_0}$ for the relation $\sim$. 
\end{cor}

\begin{proof}
As the group $G$ has a connected boundary, Corollary \ref{coneendedhypgrpsatisfddag} indicates that $X$ satisfies the condition $(\ddag_{M,K})$ for all $M>K>0$.
Besides, Corollary \ref{cnombreboutsdexh=nombreclasses} implies that if $M \geq 43\dxh + 4$ and $R_0 \geq M + \dxh$, then the number of ends of the quotient space $\xh$ is equal to the number of equivalence classes for the relation $\sim$ on the sphere $\sph{R_0}$. 
\end{proof}

To establish the existence of the sought algorithm, we need the following lemma:

\begin{lem} \label{lanneau}
Let $R_0 > 4\dxh + \dx$. If there exists a path joining two points on the sphere $\sph{R_0}$ in the complement of  $\cbxz{R_0 - 3\dxh}$ then there exists an injective path joining these points in the annulus $A(\xzbar, R_0 - 3\dxh, R_0 + 10\dx (\# \mathcal{S})^{\scriptscriptstyle R_0})$. 
\end{lem}

\begin{proof}
Let $\xbar$ and $\ybar$ be two points on the sphere $\sph{R_0}$ joined by a path of finite length in the complement of  $\cbxz{R_0-3\dxh}$. We denote by $\xbar = \overline{p_0},\overline{p_1},\ldots,\overline{p_n} = \ybar$ points on this path satisfying $d(\overline{p_i},\overline{p_{i+1}}) \leq \dx$ for $i \in \{0,\ldots,n\}$. 

As $R_0 -3\dxh -\dx  > \dxh$, any closed ball of radius $\dx$ at $\overline{p_i}$ can be lifted in $X$ using the covering map described in Proposition \ref{pcover}. Let $x$ be a lift of $\xbar = \overline{p_0}$. Firstly, lift the ball $B(\overline{p_0},\dx)$ into the ball $B(p_0,\dx)$. By construction, this ball contains a lift $p_1$ of $\overline{p_1}$ satisfying $d(p_0,p_1) \leq \dx$. By induction, we obtain lifts $p_i$ of $\overline{p_i}$ still satisfying  $d(p_i,p_{i+1}) \leq \dx$ for $i \in \{1,\ldots,n-1\}$. Denote by $y=p_n$ the lift of $\ybar = \overline{p_n}$ obtained that way.

Denote by $\proj{0}$ the $\dx$-projection of $X$ on the thickened orbit $(H\cdot x_0)^{+R_0}$ which is $2\dx$-quasi-convex. As $d(\xzbar,\xbar) = d(\xzbar,\ybar) = R_0$, we also have $d(H\cdot x_0,x) = d(H\cdot x_0,y) = R_0$ and then $\proj{0}(x)=x$ and $\proj{0}(y)=y$. Then, project the points $p_i$ on $(H\cdot x_0)^{+R_0}$. For all $i \in \{0,\ldots,n-1\}$, we have
\begin{align*}
d(\proj{0}(p_i),\proj{0}(p_{i+1})) & \leq d(p_i,p_{i+1}) + 18\dx \, \textrm{ by Proposition \ref{pcontraction},} \\
& \leq 19\dx .
\end{align*} 
This implies that
\begin{align*}
\begin{cases}
 d(\overline{\proj{0}(p_i)},\overline{\proj{0}(p_{i+1})})  \leq 19 \dx \quad \quad \forall i \in \{1,\ldots,n-1\} , \\
 d(\xbar,\overline{\proj{0}(p_{1})})  \leq 19 \dx ,\\
 d(\overline{\proj{0}(p_{n-1})},\ybar)  \leq 19 \dx .
\end{cases}
\end{align*}
Therefore, $\xbar$ is joined to $\ybar$ by a paths formed by concatenating path of length less than or equal to $19\dx$ joining the $\overline{\proj{0} (p_i)}$ with each other.
By suppressing some cycles on this path if necessary (some $\overline{\proj{0} (p_i)}$ may be equal), we obtain a path of length less than or equal to $19\dx \times \# \sph{R_0}$ joining $\xbar$ to $\ybar$ in the complement of  $\obxz{R_0 - 10\dx}$.

By definition, the cardinality of the sphere $\sph{R_0}$ is less than $(\# \mathcal{S})^{R_0}$. This implies that $\xbar$ and $\ybar$ are joined by an injective path of length less than or equal to $19\dx \times (\# \mathcal{S})^{R_0}$. 

Finally, two points on the sphere $\sph{R_0}$ joined by a path in the complement of the ball $\cbxz{R_0 - 3\dxh}$ are  joined by a path in the open annulus $A(\xzbar, R_0 - 3\dxh, R_0 + 10\dx (\# \mathcal{S})^{\scriptscriptstyle R_0})$. 
\end{proof}

\begin{thm} \label{talgorithme}
Let $G$ be a hyperbolic group with connected boundary given by a finite presentation $\langle \mathcal{S} \mid \mathcal{R} \rangle$. Let $H$ be a quasi-convex subgroup of $G$. There exists an algorithm to compute the number of relative ends of the pair $(G,H)$.
\end{thm}

\begin{proof}
The group $G$ is a hyperbolic group with connected boundary, so we can apply Corollary \ref{coneendedhypgrpnbrofends}. For $M \geq 43\dxh + 4$, there exists a constant $R_0 = M + \dxh$ such that the number of relative ends of the pair $(G,H)$ is equal to the number of equivalence classes for the relation $\sim$. Therefore, we have to determine whenever two points on $\sph{R_0}$ are joined by a path in the complement of  the open ball $\obxz{R_0-3\dxh}$.

Apply Lemma \ref{lanneau}: two points on $\sph{R_0}$ joined by a path in the complement of the ball $\cbxz{R_0 - 3\dxh}$ are joined by an injective path in the open annulus $A(\xzbar, R_0 - 3\dxh, R_0 + 10\dx (\# \mathcal{S})^{\scriptscriptstyle R_0})$.
This means that drawing the ball $\cbxz{R_0 + 10\dx(\# \mathcal{S})^{R_0}}$ allows to determine whenever two points on the sphere $\sph{R_0}$ are joined by a path in the complement of  $\obxz{R_0-3\dxh}$ and to split $\sph{R_0}$ in equivalence classes for $\sim$.

\textit{Draw the closed ball $\cbxz{R_0 + 10\dx(\# \mathcal{S})^{R_0}}$. Pick up a point $\overline{x_1}$ at distance $R_0$ from $\xzbar$ and look for every point of $\sph{R_0}$ joined to $\overline{x_1}$ by a path staying in $A(\xzbar, R_0 - 3\dxh, R_0 + 10\dx (\# \mathcal{S})^{\scriptscriptstyle R_0})$. These points form the equivalence class $[\overline{x_1}]$. If this class contains every point of the sphere $\sph{R_0}$, the pair $(G,H)$ has one relative end. Otherwise, pick up a point $\overline{x_2}$ in $\sph{R_0}\setminus{[\overline{x_1}]}$. Again, look for every point of $\sph{R_0}\setminus{[\overline{x_1}]}$ joined to $\overline{x_2}$ by a path staying in $A(\xzbar, R_0 - 3\dxh, R_0 + 10\dx (\# \mathcal{S})^{\scriptscriptstyle R_0})$. If $[\overline{x_2}] = \sph{R_0}\setminus{[\overline{x_1}]}$ then the pair $(G,H)$ has $2$ relative ends. Otherwise, pick up a point  $\overline{x_3}$ on the sphere $\sph{R_0}$ in the complement of  $[\overline{x_1}]$ and $[\overline{x_2}]$ and so on.
As the sphere $\sph{R_0}$ contains a finite number of points, the algorithm stops at some point.}

This procedure gives the number of equivalence classes for $\sim$ which is the number of relative ends of the pair $(G,H)$.
\end{proof}

If the value of $\dx$ is known, it is possible to determine explicitly every constant used in the algorithm.

%
%

\bibliographystyle{alpha}

\bibliography{references} 

\end{document}